\documentclass[12pt]{amsart}

\usepackage{amsfonts}
\usepackage{amssymb}
\usepackage{amsmath}

\DeclareMathOperator\artanh{artanh}
\DeclareMathOperator\arcosh{arcosh}

\theoremstyle{plain}
\newtheorem{theorem}{Theorem}

\newtheorem{lemma}[theorem]{Lemma}

\theoremstyle{definition}

\newcommand{\bR}{\mathbb{R}}

\begin{document}

\baselineskip 6mm

\title{A sharp inequality involving hyperbolic and inverse hyperbolic functions\footnote{It will appear in "Journal of Inequalities and Special Functions"}} 

\author{Roman Drnov\v sek}

\date{\today}

\begin{abstract}
We prove that  the inequality
$$  \cosh \left(  \arcosh(2 \cosh u) \cdot \tanh u \right) < \exp \left(  u  \cdot \tanh u \right)  $$
holds for all $u > 0$. 
We check with the computation program {\it Mathematica} that the ratio between the left-hand and the right-hand side 
is greater than 0,97 for all $u \ge 0$, so this is a quite sharp inequality.  It is also equivalent to any of the two inequalities: 
$$ \cosh \left( \sqrt{1 - \frac{1}{t^2}} \cdot \arcosh{2 t} \right) < 
 \exp \left( \sqrt{1 - \frac{1}{t^2}} \cdot \arcosh{t} \right)  $$
for all $t >1$, and  
$$ \cosh \left( c \cdot \arcosh{\frac{2}{\sqrt{1-c^2}}} \right) <
 \exp \left( c \cdot \arcosh{\frac{1}{\sqrt{1-c^2}}} \right) $$
for all $c \in (0,1)$.
\end {abstract}

\maketitle

{\it Math. Subj. Classification (2010)}: 26D07 \\

In several attemps to compute the numerical index of two-dimensional normed space equipped with an $l^p$-norm
(see \cite[Problems 2 and 3]{KMP06} or \cite[Problem 5.1]{MMP11}) 
we find a quite sharp inequality that can be added to the existing list of inequalities involving the hyperbolic functions
(see e.g. \cite{Ne12},\cite{NS11}, and \cite{NS12}). We begin with two lemmas.

\begin{lemma}
Let $x$ and $y$ be positive real numbers. Then 
\begin{equation}
\label{th}
 \tanh x \cdot \tanh y  < \tanh ( x \cdot \tanh y )  . 
\end{equation}

\end{lemma}

\begin{proof}
We will make use of the fact that the Taylor series expansion of the function $\artanh t$ has nonnegative coefficients:
$$ \artanh t = \sum_{k=0}^\infty \frac{t^{2 k + 1}}{2 k +1} \  ,  \ \  |t| < 1 . $$
Since $0 < \tanh y < 1$, we have
$$  \artanh (\tanh x \cdot \tanh y) = \sum_{k=0}^\infty \frac{(\tanh x \cdot \tanh y)^{2 k + 1}}{2 k +1} < $$
$$ <  \tanh y \cdot \sum_{k=0}^\infty \frac{(\tanh x)^{2 k + 1}}{2 k +1} = 
\tanh y \cdot \artanh (\tanh x) = x \cdot \tanh y , $$
and so \eqref{th} follows.
\end{proof}

\begin{lemma}
\label{concave}
For $0 < K < 1$, the function $\phi: [1, \infty) \to \bR$ defined by 
$$ \phi(x) = \cosh (K   \arcosh x) $$
is  strictly increasing and concave.
\end{lemma}

\begin{proof}
The first derivative 
$$ \phi^\prime (x) = \sinh (K  \arcosh x)  K  \frac{1}{\sqrt{x^2-1}} $$
is clearly positive for all $x > 1$, and so $\phi$ is a strictly increasing function.
To show that $\phi$ is concave, we must prove that the second derivative 
$$ \phi^{\prime \prime} (x) = \cosh (K  \arcosh x) \,  K^2  \frac{1}{x^2-1} - 
\sinh (K  \arcosh x)  \, K  \frac{x}{(x^2-1)^{3/2}} $$
is negative for all $x > 1$, that is 
$$ \cosh (K  \arcosh x) \,  K  \sqrt{x^2-1}  < \sinh (K  \arcosh x)  \, x . $$
Setting $u =   \arcosh x$ and $v = \artanh K$, this inequality rewrites to the inequality
$$ \tanh u \cdot \tanh v  < \tanh ( u \cdot \tanh v ) $$
that holds by \eqref{th}. This completes the proof.
\end{proof}

We now prove the main result of this paper.

\begin{theorem}
\label{thm}
We have
\begin{equation}
\label{one}
\cosh \left( \sqrt{1 - \frac{1}{t^2}} \cdot \arcosh{2 t} \right) < 
 \exp \left( \sqrt{1 - \frac{1}{t^2}} \cdot \arcosh{t} \right) 
\end{equation}
for all $t >1$, or equivalently 
\begin{equation}
\label{two}
\cosh \left( c \cdot \arcosh{\frac{2}{\sqrt{1-c^2}}} \right) <
 \exp \left( c \cdot \arcosh{\frac{1}{\sqrt{1-c^2}}} \right) =
 \exp \left( c \cdot \artanh c \right) 
\end{equation}
for all $c \in (0,1)$, or equivalently 
\begin{equation}
\label{three}
\cosh \left( \arcosh(2 \cosh u) \cdot \tanh u \right) <
 \exp \left(  u  \cdot \tanh u \right) 
\end{equation}
for all $u > 0$.
\end{theorem}

\begin{proof}
Fix $t > 1$. 
By Lemma \ref{concave}, the function $\phi: [1, \infty) \to \bR$ defined by 
$$ \phi(x) = \cosh \left( \sqrt{1 - \frac{1}{t^2}}   \arcosh x \right) $$
is  strictly increasing and concave. Therefore, its derivative 
$$ \phi^\prime (x) = \sinh \left( \sqrt{1 - \frac{1}{t^2}}   \arcosh x \right)  \sqrt{1 - \frac{1}{t^2}} \frac{1}{\sqrt{x^2-1}} ,   $$
is decreasing, and so 
$$ \max_{t \le x \le 2 t} \phi^\prime (x) = \phi^\prime (t) = \sinh \left( \sqrt{1 - \frac{1}{t^2}}   \arcosh t \right) \frac{1}{t} . $$
Now, the inequality 
$$ \phi(2 t) - \phi (t) < (2 t - t) \max_{t \le x \le 2 t} \phi^\prime (x) $$
yields that 
$$ \cosh \left( \sqrt{1 - \frac{1}{t^2}} \cdot \arcosh{2 t} \right) - 
\cosh \left( \sqrt{1 - \frac{1}{t^2}} \cdot \arcosh{t} \right) <
\sinh \left( \sqrt{1 - \frac{1}{t^2}} \cdot \arcosh{t} \right) , $$
implying the inequality \eqref{one}.

The substitution $t = \cosh u$ in \eqref{one} gives the inequality \eqref{three}, while 
the substitution $c = \tanh u$ in \eqref{three} yields the inequality \eqref{two}.
This completes the proof.
\end{proof}

Let us further explore the inequality \eqref{three}.  Since 
$$ \arcosh(2 \cosh u) = \ln (2 \cosh u+ \sqrt{4 \cosh^2 u - 1})  > $$
$$ >  \ln (2 \cosh u+ 2 \sinh u) = \ln (2 e^u) = 
\ln 2 + u , $$
the left-hand side inequality of \eqref{three} is greater than 
$$ \cosh \left( \tanh u \cdot (\ln 2 + u) \right) > \frac{1}{2} \exp \left( \tanh u \cdot (\ln 2 + u) \right) = 
2^{\tanh u - 1}  \exp \left(  u  \cdot \tanh u \right) , $$
and so we also have the inequality
\begin{equation}
\label{threeA}
2^{\tanh u - 1}  \exp \left(  u  \cdot \tanh u \right) <
\cosh \left(\arcosh(2 \cosh u) \cdot \tanh u \right) <
 \exp \left(  u  \cdot \tanh u \right) 
\end{equation}
for all $u > 0$.
Define the function $f: [0, \infty) \to \bR$ by 
$$ f(u) = \frac{\cosh \left( \arcosh(2 \cosh u) \cdot \tanh u \right) }{ \exp \left(  u  \cdot \tanh u \right)} .  $$
By \eqref{three}, we have $f(u) < 1$ for all $u > 0$, while $f(0) = 1$.
Since $\lim_{u \rightarrow \infty} \tanh u = 1$, the inequality \eqref{threeA} implies that 
$$ \lim_{u \rightarrow \infty} f(u) = 1 . $$
Furthermore, using the computation program {\it Mathematica} one can reveal a remarkable property that $f(u) > 0,972$ for all $u \ge 0$. Thus, the inequalities in Theorem \ref{thm} are surprisingly sharp.
 
\vspace{3mm}
{\it Acknowledgment.} The author acknowledges the financial support from the Slovenian Research Agency  (research core funding No. P1-0222). He also thanks the reviewer for suggestions and comments.

\vspace{2mm}

\baselineskip 5mm
\noindent
Roman Drnov\v sek \\
Department of Mathematics \\
Faculty of Mathematics and Physics \\
University of Ljubljana \\
Jadranska 19 \\
SI-1000 Ljubljana, Slovenia \\
e-mail : roman.drnovsek@fmf.uni-lj.si 


\begin{thebibliography}{9999}

\bibitem{KMP06}
V. Kadets, M. Mart\'{i}n, R. Pay\'{a}, 
Recent progress and open questions on the numerical index of Banach spaces.
RACSAM. Rev. R. Acad. Cienc. Exactas Fís. Nat. Ser. A Mat. 100 (2006), no. 1-2, 155--182. 

\bibitem{MMP11}
M. Mart\'{i}n, J. Mer\'{i}, M. Popov, 
On the numerical radius of operators in Lebesgue spaces.
J. Funct. Anal. 261 (2011), no. 1, 149--168.

\bibitem{Ne12}
E. Neuman, Refinements and generalizations of certain inequalities involving trigonometric
and hyperbolic functions. Adv. Inequal. Appl.  1 (2012), 1--11.

\bibitem{NS11}
E. Neuman, J. S\'{a}ndor, Optimal Inequalities For Hyperbolic And Trigonometric
Functions.  Bull. Math. Anal. Appl. 3 (2011), no. 3,177--181.

\bibitem{NS12}
E. Neuman, J. S\'{a}ndor, Inequalities for hyperbolic functions.  Appl. Math. Comput. 218 (2012), no. 18,
9291--9295.


\end{thebibliography}
\end{document}